\documentclass[letterpaper,reqno]{amsart}

\usepackage{amsmath} 
\usepackage{amsfonts} 
\usepackage{amssymb} 
\usepackage{epsfig} 
\usepackage{graphicx} 
\usepackage{multirow} 
\usepackage{verbatim} 
\usepackage{rotating} 
\usepackage[all]{xy} 
\usepackage{appendix}

\newtheorem{theorem}{Theorem}[section]
\newtheorem{lemma}[theorem]{Lemma}

\theoremstyle{definition}
\newtheorem{definition}[theorem]{Definition}

\theoremstyle{remark}
\newtheorem{remark}[theorem]{Remark}
\theoremstyle{notation}

\numberwithin{equation}{section}
\theoremstyle{corollary}
\newtheorem{corollary}[theorem]{Corollary}

\usepackage[bookmarks=false]{hyperref}
\newcommand{\Map}{\mathrm{Map}}
\newcommand{\map}{\mathrm{map}}

\newcommand{\Cat}{\mathsf{Cat}}

\newcommand{\Top}{\mathsf{Top}}
\newcommand{\Chk}{\mathsf{dgMod}_{k}}

\newcommand{\Chr}{\mathsf{dgMod}_{R}}
\newcommand{\Chs}{\mathsf{dgMod}_{S}}
\newcommand{\Chrs}{\mathsf{dgMod}_{R-S}}
\newcommand{\Chrr}{\mathsf{dgMod}_{R-R}}
\newcommand{\C}{\mathsf{C}}
\newcommand{\RR}{\mathsf{R}}
\newcommand{\SSS}{\mathsf{S}}

\newcommand{\HH}{\mathrm{HH}_{k}}
\newcommand{\Derive}{\mathrm{Der}}
\newcommand{\Ext}{\mathrm{Ext}}

\newcommand{\coh}{\mathrm{H}}

\newcommand{\Ho}{\mathrm{Ho}}

\newcommand{\Algk}{\mathsf{dgAlg}_{k}}
\newcommand{\Algr}{\mathsf{dgAlg}_{R}}

\newcommand{\Algrs}{\mathsf{dgAlg}_{R-S}}
\newcommand{\CAlgk}{\mathsf{dgCAlg}_{k}}

\begin{document}

\title[]{The mapping space of unbounded differential graded algebras}

\author[]{Ilias Amrani}
\address{Department of Mathematics, Masaryk University\\ Kotlarska 2\\ Brno, Czech Republic.}
\email{ilias.amranifedotov@gmail.com}
\email{amrani@math.muni.cz}
\thanks{Supported by the project CZ.1.07/2.3.00/20.0003
of the Operational Programme Education for Competitiveness of the Ministry
of Education, Youth and Sports of the Czech Republic.
}

\thanks{}

\subjclass[2000]{Primary 55, Secondary 14 , 16, 18}




\begin{abstract}
In this paper, we give a concrete description of the higher homotopy groups ($n>0$) of the mapping space $\Map_{\Algk}(R,S)$ for $R$ and $S$ unbounded differential graded algebras (DGA) over a commutative ring $k$. In the connective case,  we describe the relation between the higher (negative) Hochschild cohomology $\HH^{-n+1}(R,S)$  and higher homotopy groups $\pi_{n}\Map_{\Algk}(R,S)$, when $n>1$.  
\end{abstract}

\maketitle
\section*{Introduction}
Our work is based on the recent paper \cite{DH2010}. Given a (symmetric) monoidal model category $(\C,\otimes)$ and a compatible model structure on the category of monoids $\C^{\otimes}$ (with underlying weak equivalences and fibrations of $\C$). Suppose that we have a Dwyer-Kan model structure on the category of small $\C$-enriched categories $\Cat_{\C}$ (weak equivalences are homotopy enriched fully faithful and homotopy essentially surjective functors). The main idea is that the mapping spaces $\Map$ of these three model structures are closely related under some assumptions (\textbf{six Axioms} (\ref{axiom6}) described in \cite[section 3]{DH2010}). The none obvious axiom is the third one \ref{axiom3}. Our paper is a concrete application of this idea in the setting, where
\begin{itemize}
\item $\C$ is the symmetric monoidal category of \textbf {unbounded differential graded modules} (DG modules for short) over a commutative ring $k$ \cite{Hovey}. 
\item $\C^{\otimes}$ is the model category of unbounded DG $k$-algebras \cite{SS}.
\item $\Cat_{\C}$ is the model category of (small) DG categories, denoted by $dg$-$\Cat$ \cite{tabuada2005structure}. 
\end{itemize} 

Before going to our main point, we illustrate the previous idea in the \textit{topological} setting when $\C=\Top$. There is a well known fiber sequence (Cf. \cite[p. 6]{DH2010})  for given topological groups $G$ and $H$, which is
\begin{equation}\label{fibration}
\map_{\ast}(BH,BG)\rightarrow \map(BH,BG)\rightarrow BG,
\end{equation}
where $BG$ and $BH$ are the classifying spaces of $G$ and $H$. Notice that if $G$ and $H$ are discrete groups, then $\pi_{0}\map(BH,BG)=\mathrm{Rep}_{G}(H)$ is the set of equivalence classes of representations of $H$ in $G$.
The interpretation in the categorical setting is as follows. Let $\Cat_{\Top}$ be the model category of small topological categories \cite{Amrani1}, where weak equivalences are Dwyer-Kan equivalences. Denote by $\Top^{\otimes}$ the model category of topological monoids. The previous fiber sequence (\ref{fibration}) has a model-categorical translation in terms of mapping spaces:
\begin{equation}\label{master}
\Map_{\Top^{\otimes}}(H,G)\rightarrow \Map_{\Cat_{\Top}}(\mathsf{H},\mathsf{G})\rightarrow \Map_{\Cat_{\Top}}(\ast,\mathsf{G}),
\end{equation}
where $\mathsf{G}$ (resp. $\mathsf{H}$) is the topological category with one object and endomorphism monoid $G$ (resp $H$). 
Here, we had supposed that $G$ and $H$ are topological groups. The fiber sequence (\ref{master})  is still valid for topological monoids. and coincides with \ref{fibration} in the case where $G$ and $H$ are topological groups.  \\
\textbf{The goal} in this paper is the construction of fiber sequence (\ref{master}) in the unbounded differential graded setting (Cf. \ref{Hess2}), namely
\begin{equation}\label{master1}
\Map_{\Algk}(R,S)\rightarrow\Map_{dg-\Cat}(\RR,\SSS)\rightarrow\Map_{dg-\Cat}(\mathsf{k},\SSS).
\end{equation}
\section*{Notations}
In what follows, all model structures are defined by taking homology isomorphisms for weak equivalences and degree-wise surjections for fibrations. 
\begin{itemize}
\item $k$ is a fixed commutative ring of any characteristic.
\item $\Chk$ the stable symmetric monoidal closed model category of differential graded $k$-modules. Our convention is the cohomological gradation i.e., the differentials increase the degree by  +1.  
\item The (\textbf{derived when necessary}) tensor product over $k$ of differential graded $k$-modules is denoted by $\otimes$. 
\item $\Algk$ is the model category of unbounded differential graded $k$-algebras \cite{SS} i.e., the category of monoids in $\Chk$. 
\item $\Algr$ (differential graded $R$-algebras) is the model category of graded differential $k$-algebras under a fixed graded differential $k$-algebra $R$, i.e., objects are morphisms $R\rightarrow A$ in $\Algk$.  
\item $\Algrs$ is the model category of graded differential $k$-algebras under a fixed graded differential $k$ algebras $R$ and $S$ i.e. objects are pairs of morphisms $R\rightarrow A,~ S\rightarrow A$ in $\Algk$. 

\item For any differential graded $k$-algebras $R$ and $S$ we denote by $\Chrs$ the stable model category of differential graded $R-S$-bimodules. The category  $\Chrs^{0}$ is the model category of pointed DG $R-S$-modules i.e., objects are coming with an extra map $k\rightarrow M$. 
\item We denote the \textbf{derived} mapping space of a model category by $\Map$.  
\item The n-th homology group of a differential graded $k$-algebra $R$ is denoted by $\coh^{n}R$ 
\item  The suspension functor $\Sigma: \Chr\rightarrow \Chr$ is defined as follows $(\Sigma M)_{n}=M_{n+1}$. Obviously, this functor has an inverse denoted by $\Sigma^{-1}$. 
\item Let $R\in\Algk$, we denote the derived category of $R$ by $\mathsf{D}_{R}$ which is the homotopy category of DG $R$-modules, i.e., $\Ho(\Chr)$. For more details Cf. \cite{Hovey}.

\end{itemize}

\subsection{Ext functor and Hochschild cohomology }
In this paragraph, we recall a well known translation between notions defined in \textit{Algebraic Geometry} and \textit{Algebraic Topology}. We use the same conventions as in \cite{gelfand2003methods}.
Let $R\in \Algk $, and $M, N\in\Chr $. In the stable model category of $\Chr$  \cite{Hovey, krause2007derived} and for $n\in \mathbb{Z}$,
\begin{eqnarray*} \label{hochschild}
\Ext^{n}_{R}(N,M) & \simeq &\mathsf{D}_{R}(N,\Sigma^{n} M)\\
&\simeq & \Ho(\Chr)(N,\Sigma^{n} M)\\
&\simeq& \pi_{0}\Map_{\Chr}(N,\Sigma^{n}M).
\end{eqnarray*}
\begin{remark}
Our gradation is the same as in \cite{krause2007derived} and opposite to the one used in \cite{Hovey}.
\end{remark}
\begin{remark}\label{remarque1} 
Recall that the model category $\Chr$ is naturally pointed, and the functor $\Sigma^{-1}$ is the loop functor. Therefore, it  follows by \cite[Lemma 6.1.2]{Hovey}, that for any $n\geq 0$ there is a weak homotopy equivalence of pointed simplicial sets, where the base point is the zero morphism,
$$\Map_{\Chr}(M,\Sigma^{-n}N)\sim \Omega^{n}\Map_{\Chr}(M,N).$$ 
For $n\geq 0$, we have the following group isomorphisms: 
 \begin{eqnarray*} \label{Hochschild1}
\Ext^{-n}_{R}(N,M) &\simeq & \pi_{0}\Map_{\Chr}(N,\Sigma^{-n} M)\\
&\simeq & \pi_{n}\Map_{\Chr}(N,M).
\end{eqnarray*} 
\end{remark}
\begin{remark}
If $R$ is a $k$-algebra and $M,~N$ are any two $R$-modules i.e., $R,M$ and $N$ are DG modules concentrated in degree 0, then $\Ext^{n}_{R}(M,N)=0$ for $n<0$  (Cf. \cite{krause2007derived}).     
\end{remark}
\begin{definition}\label{Hochschild3}
Let $R\in \Algk$  and let $M$ be a DG $R$-bimodule. The Hochschild cohomology of $R$ with coefficient in $M$ is defined for all $n\in\mathbb{Z}$, by
\begin{eqnarray*}
\HH^{n}(R,M) &\simeq& \Ext^{n}_{R\otimes R^{op}}(R,M)\\
&\simeq &\pi_{0}\Map_{\Chrr}(R,\Sigma^{n}M). 
\end{eqnarray*}
If $M=R$, we denote the Hochschild cohomology of $R$ with coefficient in $R$ simply by $\HH^{\ast}(R)$.   
\end{definition}
\begin{remark}
The correct definition for the Hochschild cohomology $\HH^{n}(R,M)$ is  $\Ext^{n}_{R\otimes^{\mathbf{L}} R^{op}}(R,M)$, but we took the liberty to denote the derived tensor product as an ordinary tensor product!
\end{remark}
\begin{remark}\label{remarque2}
For any DG $R$-module $N$, we recall that (Cf. \cite{krause2007derived})
$$\pi_{n}\Map_{\Chr}(R,N)_{\ast}\simeq\mathsf{D}_{R}(R,\Sigma^{-n}N)\simeq\coh^{-n}(N).$$
\end{remark}
\begin{remark}
When the homotopy groups are computed without mentioning the base point, it will mean that the base point is the null morphism. 
\end{remark}
For any DG algebra $R$ and any $R$-bimodule M, the usual definition of the Hochschild cohomology is given by $\HH^{\ast}(R,M)=\coh^{\ast}\mathrm{Hom}_{R\otimes R^{op}}(R,M)$ (e.g. \cite[3.14]{dugger2007topological}), where $\mathrm{Hom}_{R\otimes R^{op}}(R,-):\Chrr\rightarrow \Chk$, is the right (derived) functor having as left (derived) adjoint $R\otimes -$. If $n\geq 0$, applying \cite[Theorem 2.12]{DH2010}, we obtain the following group isomorphisms  
\begin{eqnarray*}
\coh^{-n}\mathrm{Hom}_{R\otimes R^{op}}(R,M) &\simeq &\pi_{n}\Map_{\Chk}(k,\mathrm{Hom}_{R\otimes R^{op}}(R,M))\\
&\simeq & \pi_{n}\Map_{\Chrr}(R,M)\\
&\simeq & \mathrm{Ext}^{-n}_{R\otimes R^{op}}(R,M).
\end{eqnarray*}
\begin{remark}
In order to be clear, by derived functor of $\mathrm{Hom}_{R\otimes R^{op}}(R,M)$ we mean the right derived functor $\mathbf{R}\mathrm{Hom}_{R\otimes R^{op}}(R,M)$.
\end{remark}
\section{Main results}
We start by fixing a morphism $\phi: R\rightarrow S$ in $\Algk$ (such that $S$ is a strict DG $k$-algebra \ref{lemma1bis}). Our main result concerns the higher homotopy groups of the mapping space
$$\pi_{n}\Map_{\Algk}(R,S)_{\phi}:=[R,S]_{n}^{\otimes}~\textrm{for}~n>1.$$
We give an explicit long exact sequence relating these higher homotopy groups with  $\coh^{\ast}S$ and (negative) Hochschild  cohomology $\HH^{\ast}(R,S)$. Moreover, we study  the case  
  $$\pi_{1}\Map_{\Algk}(R,R)_{id}:=[R,R]_{1}^{\otimes}.$$
\textbf{Theorem A} (cf. \ref{th2})\\
\textit{Let R in $\Algk$ (a strict DG $k$-algebra \ref{lemma1bis}). There is an exact sequence of groups
$$\dots\rightarrow\coh^{-1} (R)\rightarrow  [R,R]_{1}^{\otimes}\rightarrow \HH^{0}(R)^{\star}\rightarrow \coh^{0}(R)^{\star},$$
where $\HH^{0}(R)^{\star}$ and $\coh^{0}(R)^{\star}$ are the groups of units in the rings $\HH^{0}(R)$ and $\coh^{0}(R)$. }\\\\
\textbf{Theorem B} (cf. \ref{th1})\\
\textit{Let $\phi: R\rightarrow S$ be a morphism in $\Algk$  (such that $S$ is a strict DG $k$-algebra \ref{lemma1bis}). There is an exact sequence of abelian groups  
$$\coh^{-1}(S)\leftarrow \HH^{-1}(R,S)\leftarrow [R,S]^{\otimes}_{2}\leftarrow \coh^{-2}(S)\leftarrow \HH^{-2}(R,S)\leftarrow [R,S]^{\otimes}_{3}\leftarrow \dots  $$  
where $S$ is seen as an $R$-bimodule via $\phi$.}\\\\
\textbf{Lemma C} (cf \ref{derivative})\\
 \textit{If $R$ is a connective DG $k$-algebra, $R\oplus M$ a connective, square-zero extension, with $\phi:R\rightarrow R\oplus M$ the obvious inclusion, then  for all $n>1$
$$\Derive^{-n}_{k}(R,M)\oplus \HH^{-n+1}(R)\simeq \pi_{n}\Map_{\Algk}(R,R\oplus M)_{\phi}.$$}


\section{The six Axioms}\label{axiom6}
We verify the six Axioms described in \cite[section 3]{DH2010}, for the following categories $\Chk,~ \Chrs$ and $\Algk$. These Axioms will be proved and defined in details (essentially the third Axiom \ref{ax3}) the rest are more or less obvious in our setting. 
\subsection{Axiom I}\cite[3.1]{DH2010}
The model structures on $\Chs, ~ \Chrs$ and $\Algk$ are all compatible in the sense that the weak equivalences and fibrations are the underlying weak equivalences and fibrations  in $\Chk$. Hence, there is nothing to verify. 
\subsection{Axiom II}\cite[3.2]{DH2010}
Let $\Chrs^{0}$ denote the category of pointed DG $R-S$-modules $X$ i.e., coming with a morphism $k\rightarrow X$ in $\Chk$. The second axiom requires the existence of a Quilen adjunction  
\begin{equation}\label{adj}
\xymatrix{ \Chrs^{0}  \ar@<2pt>[r]^{ F} & \Algrs  \ar@<2pt>[l]^{U}}
\end{equation}
since all the involved categories are locally presentable \cite[proposition 3.7]{shipley2007hz}. The forgetful functor $U$ commutes with limits and directed colimits, therefore the left adjoint exists by \cite[p. 65]{adamek1994locally}. The existence of a model structure on $ \Chrs^{0}$ (where weak equivalence and fibrations are underlying weak equivalences and fibrations of $\Chrs$) is guaranteed by \cite[Proposition 1.1.8]{Hovey}. Moreover, it is a Quillen adjunction since fibrations and weak equivalences are those of the underlying category $\Chk$, by definition. We denote the image of the unit  $1\in k\rightarrow X$ also by $1\in X$. 
Now, we give a concrete description of the functor $F$. 
\begin{definition}\label{defF}
The DG $R-S$-algebra $F(X)$ is the quotient of the free DG $k$-algebra 
$$ T(X)=k.1\oplus X\oplus X\otimes X\oplus X^{\otimes^{3}}\otimes\dots$$
subject to the following relations:
\begin{enumerate}
\item For any $n\in\mathbb{N}^{\ast}$, $\underbrace{1\otimes 1\otimes \dots \otimes 1}_{\text{n times}}\sim 1$ and the differential of 1 is 0. 

 \item For any $s\in S$  and any  $x_{1}\otimes\dots \otimes x_{i-1}\otimes x_{i+1}\dots\otimes x_{n}\in T(X)$
$$x_{1}\otimes\dots \otimes x_{i-1}\otimes 1.s\otimes x_{i+1}\dots\otimes x_{n}\sim x_{1}\otimes\dots \otimes x_{i-1}.s\otimes x_{i+1}\dots\otimes x_{n}.$$
\item For any $r\in R$ and any $x_{1}\otimes\dots \otimes x_{j-1}\otimes x_{j+1}\dots\otimes x_{m}\in T(X)$
$$x_{1}\otimes\dots \otimes x_{j-1}\otimes r.1\otimes x_{j+1}\dots\otimes x_{m}\sim x_{1}\otimes\dots \otimes x_{j-1}\otimes r.x_{j+1}\dots\otimes x_{m}.$$
\end{enumerate} 
\end{definition}
\begin{remark}
The first relation (1) of the previous definition \ref{defF} is actually redundant.   
\end{remark}
Recall that the differentials of $T(X)$ are given by (the sign depends on the degree of elements) 
$$ d(x_{1}\otimes\dots \otimes x_{i}\otimes x_{i+1}\dots\otimes x_{n})=\sum_{i=1}^{n}\pm ~x_{1}\otimes\dots \otimes dx_{i}\otimes x_{i+1}\dots\otimes x_{n}.$$ 
We can define the following morphisms in $\Algk$ by universal property of $F(X)$
\begin{itemize}
\item The  morphism $S\rightarrow F(X)$ takes $s$ to $1\otimes 1.s$. 
\item The morphim $R\rightarrow F(X)$ takes $r$ to $r.1\otimes 1$.
\item Notice that in $F(X)$ we have by definition $1\otimes 1.s = 1.s = 1.s\otimes 1$ for any $s\in S$ and similarly  $1\otimes r.1 = r.1 = r.1\otimes 1$ for any $r\in R$. 
\end{itemize}

\begin{lemma}\label{universal}
The functor $F: \Chrs^{0}\rightarrow \Algrs$ is a left adjoint of the forgetful functor $U$. 
\end{lemma}
\begin{proof}
In order to prove that $F$ is a left adjoint, we check the universal property i.e., given a morphism 
$f: X\rightarrow U(A)$ in $\Chrs^{0}$ where $A\in  \Algrs$, there is a unique extension 
$\overline{f}:F(X)\rightarrow A$ of DG $R-S$-algebras. By definition, the element $1\in X$ goes to the unit 
$e$ of  $A$, such that $f(r.1)= r.e$ and $f(1.s)=e.s$. The equivalence class of the tensor element $x_{1}\otimes x_{2}\dots \otimes x_{n}$ in $F(X)$  
is sent by $\overline{f}$ to $f(x_{1}).f(x_{2})\dots f(x_{n})$. This morphism is well defined since $f$ is a map of right DG $S$-modules. Thus, $x_{1}\otimes x_{j}\otimes 1.s\otimes x_{j+2}\dots x_{n}$ and $x_{1}\otimes\dots x_{j}.s\otimes x_{j+2}\dots x_{n}$  have the same image. By analogy, elements of the form $x_{1}\dots\otimes x_{i}\otimes r.1\otimes x_{i+2}\dots x_{n}$ and $x_{1}\dots\otimes x_{i}\otimes r.x_{i+2}\dots x_{m}$ have the same image by $\overline{f}$. Hence, $\overline{f}$ is uniquely defined. Moreover, any map
$\overline{g}:F(X)\rightarrow A$ defines, obviously, a unique map of DG $R-S$-bimodules $g:X\rightarrow UA$. Therefore, there is an isomorphism of sets 
$$\Chrs^{0}(X, UA)\simeq \Algrs(F(X),A).$$
\end{proof}
\begin{remark}
Our construction of the functor $F$ was inspired by a topological analogy. The adjunction between the category of \textbf{pointed} topological spaces and the category of topological monoids is given by the forgetful functor  and James's functor as left adjoint, we refer to \cite{Vogt}.  
\end{remark}

\subsection{Axiom IV, V and VI}\cite[3.5, 3.6, 3.8]{DH2010}
These axioms are easy to verify. For the fourth Axiom, it is enough to take $R$ cofibrant in $\Algk$, while the fifth Axiom holds if $S$ is cofibrant in $\Chs$, which is trivial since $k$ is cofibrant in $\Chk$.  
The last Axiom requires that the two maps defined below are weak equivalences.
\begin{itemize}
\item We need $k^{c}\otimes S\rightarrow k\otimes S $ to be an equivalence of right $S$-modules, where $k^{c}$ is some cofibrant replacement of $k$ in $\Chk$. This is satisfied, since we can take $k^{c}=k$. 
\item For any map $R\rightarrow S$ in $\Algk$, let $R^{c}$ be a cofibrant replacement of $R$ in the category of DG $R$-bimodules. The map $R^{c}\otimes_{R} S\rightarrow R\otimes_{R}S\simeq S$ is a weak equivalence of $R-S$-modules (in $\Chk$ in fact). In order to prove the statement, we use the a concrete model for the cofibrant replacement $R^{c}$, which is given by the Bar construction $B(R)$. Since $R$ is cofibrant as DG $R$-module, then by  
\cite[Proposition 7.5]{EKMM}, the natural map $B(R,R,S)\rightarrow R\otimes_{R}S\simeq S$ is a weak equivalence in $\Chrs$. On the other hand, $B(R,R,S)$ is naturally isomorphic to $B(R)\otimes_{R} S$. We conclude that the morphism $R^{c}\otimes_{R}S\rightarrow S$ is an equivalence of DG $R-S$-bimodules. 
\end{itemize}  
 

 \subsection{Axiom III}\label{axiom3} 
\begin{definition}\label{defff} \cite[3.4]{DH2010} A \textbf{distinguished object} in $\Chrs^{0}$ is a pointed DG $R-S$-bimodule, such that $k\rightarrow X$ induces an equivalence $S\simeq k\otimes S\rightarrow X$ of right $S$-module. An object $A$ in  $\Algrs$ is said to be \textit{distinguished} if the map induced by the unit $S\simeq k\otimes S\rightarrow A$ is a weak equivalence. 
\end{definition}
\begin{definition}\label{ax3}[\textbf{Axiom III}]
We say that the functor $F:\Chrs^{0}\rightarrow \Algrs$ verifies the third axiom if it sends cofibrant distinguished objects in $\Chrs^{0}$ to (cofibrant) distinguished object in $\Algrs$.
\end{definition}
\begin{lemma}\label{lemma1}
Let $R$ be a cofibrant DG $k$-algebra and $M$ be a cofibrant DG $R-S$-bimodule such that $M$ is zigzag equivalent to $S$ as a right DG $S$-module. Then there is a map $\phi: R\rightarrow S$ of DG-algebras and a weak equivalence $M\rightarrow S$ of DG $R-S$-bimodules (where the left action of $R$ on $S$ is induced by $\phi$). 
\end{lemma} 
\begin{proof}
The proof of this lemma is based on To\"en's fundamental theorem \cite[Theorem 4.2]{toen}. To\"en's theorem compares two models for the mapping space of the model category of dg-categories.
Since $R$ is a cofibrant DG-algebra, then $R$ is a cofibrant DG-module \cite[Proposition 2.3]{toen}. By \cite[proposition 3.3]{toen}, a cofibrant DG $R-S$-module $M$ is a cofibrant DG $S$-module (forgetting the DG $R$-module structure). It follows by \cite[Theorem 4.2]{toen} that each cofibrant $R-S$-bimodule $M$ is zigzag equivalent to $S$ where the the left action of $R$ on $S$ is given by some map of DG-algebras $\phi: R\rightarrow S$. More precisely, we have the following zigzag of weak equivalences of $R-S$-bimodules
$$ S\leftarrow M_{1}\leftrightarrow M_{2}\dots\rightarrow M,$$
where $M_{i}$ are cofibrant as DG $S$-modules. We replace functorially by $M_{i}^{c}$ (cofibrant replacement in the category of $R-S$-bimodules, and hence we obtain a weak equivalence (not unique) of DG $R-S$-bimodules $M\rightarrow S$.      
\end{proof}
\begin{remark}\label{rem1}
Under the same hypothesis as in lemma \ref{lemma1}, if in addition $M$ is pointed (i.e. $k\rightarrow M$) then $M\rightarrow S$ is an equivalence of pointed DG $R-S$-bimodules.  
\end{remark}

\begin{definition}\label{lemma1bis}
Let $S$ be a DG $k$-algebra, we say that $S$ is a \textbf{strict} if the differential $d_{-1}: S_{-1}\rightarrow S_{0}$ is identically 0. 
\end{definition}
\begin{definition}\label{}
Let $S$ be a DG $k$-algebra and $1$ the unit element, a homotopy invertible element $x\in S_{0}$ is a cocycle such that there exists an other cocycle $y\in S_{0}$ with the property that $xy-1$ and $yx-1$ are boundaries 
\end{definition}
\begin{remark}
If $S$ is a strict DG $k$-algebra, then any homotopy invertible element is strictly invertible. 
\end{remark}
\begin{lemma}\label{inverse}
Suppose that $S$ is DG algebra where all homotopy invertible elements are strictly invertible. Then any weak equivalence $S\rightarrow S$ in $\Chs$ is an isomorphism.  
\end{lemma}

\begin{proof}
Any $S$-linear morphism $f:S\rightarrow S$ is determined by the image of the unit 1. Since $f$ is a weak equivalence and and $S$ is fibrant cofibrant in $\Chs$, implies that $f$ has a homotopy inverse $g$. Hence $fg(1)-1$ and $gf(1)-1$ are boundaries. But by hypothesis on $S$, we have that $fg(1)-1=gf(1)-1 =0$  
\end{proof}
\begin{remark}\label{remarqueq}
Till the end of the subsection, we will assume that $R$ is a cofibrant DG algbera and $S$ is a strict  DG $R$-algebra \ref{lemma1bis}.  
\end{remark}
\begin{lemma}\label{S-eq}
Let $S$ a DG algebra as in \ref{remarqueq}, then the universal map of DG algebras $S\rightarrow F(S)$ is an isomorphism.   
\end{lemma}
\begin{proof}
Recall that $R\in \Algk$ and we have a map of DG algebras $\phi: R\rightarrow S$ such that  all homotopy invertible elements of $S$ are strictly invertible. Take a representative element in $F(S)$ of the form $s_{1}\otimes\dots \otimes s_n$.  Since the chosen element $1$ (not the unit  in general) of $S$ is strictly invertible, the element $s_{1}\otimes\dots \otimes s_n$ can be reduced to an element of $S$ by using only relations (1) and (2) in \ref{defF}. More precisely 
$$s_{1}\otimes s_2\dots \otimes s_n= s_{1}\otimes1.1^{-1}s_2\dots \otimes 1.1^{-1}s_n\sim s_1.1^{-1}s_{2}\dots 1^{-1}s_n=~s.$$
Hence, the map $S\rightarrow F(S)$ is an isomorphism.
\end{proof}
\begin{definition}\label{seconddef}
Let $I: \Chrs^{0}\rightarrow \Chrs$ be the functor defined as follows: 
$I(X)$ is a two sided ideal of $T(X)$ generated by the relations:
\begin{enumerate}
\item $x\otimes 1.s - x.s$ for any element $s\in S$ and any element $x\in X$ where 1 is the image of the unit $k\rightarrow X$.
\item $r.1\otimes y -r.y$ for any element $r\in R$ and any element $y\in X$ 
\end{enumerate} 
By definition \ref{defF}, it is clear that the quotient in $\Chrs$ or $ \Chk$ of $T(X)$ by $I(X)$ is isomorphic to $F(X)$. 
\end{definition}
\begin{lemma}\label{lemmaimp}
Let $S$ be a strict DG $k$-algebra and let $\phi: R\rightarrow S$ a map in $\Algk$ which induces a left action of $R$ on $S$. Let $e$ be the unit of $S$ and fix an invertible element $1\in S$. Then $\coh^{\ast}I(X)$ is generated by $1\otimes 1-1$ as $\coh^{\ast}R \otimes \coh^{\ast}T(S)-\coh^{\ast}T(S)\otimes \coh^{\ast}S$ graded bimodule.      
\end{lemma}
\begin{proof}
By definition $I(S)$ is generated by elements  $x\otimes 1.s - x.s$  and $r.1\otimes y -r.y$ as DG $T(S)$-bimodule \ref{seconddef}.
Since $1$ is invertible element, the second kind of generators  $r.1\otimes y -r.y$ can be reduced to the first kind of generators, more precisely 
$$r.1\otimes y -r.y=\phi(r).1\otimes y -\phi(r).y= \phi(r).1\otimes 1.1^{-1} y -\phi(r).1.1^{-1}.y$$
which is of the form $x\otimes 1.s - x.s$. 
Now, we suppose that $\phi:R\rightarrow S$ is a fibration i.e. a surjective map, then any generator of the form $x\otimes 1.s - x.s$ can be written as $x.1^{-1} (1\otimes 1-1) s$. In this case  $\coh^{\ast}I(X)$ is generated by $1\otimes 1-1$ as $ \coh^{\ast}R \otimes \coh^{\ast}T(S)-\coh^{\ast}T(S)\otimes \coh^{\ast}S$ graded bimodule. 
Now, if $\phi:R\rightarrow S$ is not a fibration, we factor $\phi$ as a trivial cofibration followed by a fibration  
$$\xymatrix{R\ar[r]^{\sim}_{\phi_{1}} & R^{'}\ar[r] _{\phi_{2}}& S},$$
Since the generators of $I(S)$ are of the form $x\otimes 1.s - x.s$ (do not depend on the action of $R$ on $S$), we conclude that $\coh^{\ast}I(S_{\phi})$ and $\coh^{\ast}I(S_{\phi_{2}})$  for both actions $\phi: R\rightarrow S$ or $\phi_{2}: R\rightarrow S$ are isomorphic.
By definition, we have that $\coh^{\ast}R^{'}\rightarrow \coh^{\ast}R $ is an isomorphism of graded algebras, we conclude that $H^{\ast}I(S)$ is generated by $1\otimes 1-1$ as  $ \coh^{\ast}R \otimes \coh^{\ast}T(S)-\coh^{\ast}T(S)\otimes \coh^{\ast}S$ graded bimodule for any map of DG algebras $\phi:R\rightarrow S$. 
\end{proof}

\begin{lemma}\label{lemma2}
Let $M$ be a cofibrant distinguished object  in $\Chrs^{0}$ (cf \ref{defff}) and $f: M\rightarrow S$ be an equivalence in $\Chrs^{0}$ (cf.\ref{lemma1}), then $F(M)\rightarrow F(S)$ is an equivalence in $\Algrs$. 
\end{lemma}
\begin{proof}
First of all, since $M$ is cofibrant we can reduce the problem to a trivial fibration of $R-S$-bimodules, namely we factor the pointed weak equivalence in $\Chrs^{0}$ as 
$ M\rightarrowtail X\twoheadrightarrow S$. It is sufficient to prove that $F(X)\rightarrow F(S)$ is a weak equivalence. Moreover, the point $\nu: k\rightarrow S$ induces an isomorphism of DG right $S$-modules since $S$ is distinguished (because $X$ is distinguished by assumption) and has the property of the remark \ref{remarqueq} and \ref{lemma1bis}.\\
The morphism $f: M\rightarrow S$ (which is a trivial fibration in $\Chrs^{0}$) has a section $g$ in the category $\Chs^{0}$ because $k\rightarrow S$ induces an isomorphism $ k\otimes S\rightarrow S$ in $\Chs$ since the chosen element in $S$ is strictly invertible, more precisely, we have 
$$\xymatrix{ M\ar[r]^{f} & S\ar[r]^{b}_{\simeq} & S}$$ 
where $b$ is the isomorphism of $\Chs$ that takes the base point 1 of $S$ to the unit $e$ of the DGA $S$ (cf \ref{inverse}), it follows that $b\circ f$ has a section in $\Chs$ taking $e$ to the base point of $M$, hence $f$ has a section $g$ in $\Chs^{0}$ taking $1\in S$ to the base point of $M$, i.e., $f\circ g=id_{S}$. 

 This implies that the two sided ideal $I(S)$ is generated only by  elements of the form $x\otimes 1.s - x.s$. Hence, the morphism $g$ induces a  section $\overline{\overline{Tg}}~$ of $~\overline{\overline{Tf}}:I(S)\rightarrow I(X)$ such that  $\overline{\overline{Tg}}$ $\overline{\overline{Tf}}$ are restrictions of $Tg$ and $Tf$. 

For any trivial fibration of pointed DG $R-S$ module $f:X\twoheadrightarrow S$, we have the following commutative diagram:
 \def\cartesien{%
    \ar@{-}[]+R+<6pt,-1pt>;[]+RD+<6pt,-6pt>%
    \ar@{-}[]+D+<1pt,-6pt>;[]+RD+<6pt,-6pt>%
  }
$$
  \xymatrix{
    I(X) \ar@/_/[rddd]_-{\overline{\overline{Tf} }}\ar@/^/[rrrd] ^{i_3}\ar@{.>}[rd]_{i} \\
    & C \ar[rr]^{i_{1}}\ar@{>>}[dd]^{\sim}_{\overline{Tf}} \ar@/_/@{.>}[ul]_{h} \cartesien& & T(X) \ar@{>>}[dd]^{\sim}_{Tf} \\
    & & & \\
    & I(S) \ar[rr]^{i_{2}}\ar@/_/@{.>}[uu]_{\overline{Tg}} \ar@/_/@{.>}[uuul]_{\overline{\overline{Tg}}}& & T(S)\ar@/_/@{.>}[uu]_{Tg} 
  }
$$
satisfying the fallowing relations:
\begin{itemize}
\item $Tf\circ Tg=id_{T(S)}$.
\item $\overline{Tf}: C\rightarrow I(S)$ and $\overline{Tg}: I(S)\rightarrow C$  are weak equivalences and $\overline{Tf}\circ \overline{Tg}=id_{I(S)}$.

\item $\overline{\overline{Tf}}:I(S)\rightarrow I(X)$ and $\overline{\overline{Tg}}: I(S)\rightarrow I(X)$ verify  $\overline{\overline{Tf}}\circ \overline{\overline{Tg}}=id_{I(S)}$.
\item $Tf\circ i_1= i_2\circ \overline{Tf}$.
\item $i_{1}\circ i=i_3$.
\item $\overline{Tf}\circ i=\overline{\overline{Tf}}$.
\item Define $h=  \overline{\overline{Tg}}\circ \overline{Tf}: C\rightarrow I(X)$. 
\end{itemize}

Applying the cohomology functor to the previous pullback, we obtain a pullback diagram in the category of pointed $\coh^{\ast}R\otimes \coh^{\ast}T(S)-\coh^{\ast}T(S)\otimes\coh^{\ast}S $ graded bimodules:
$$
  \xymatrix{
   \coh^{\ast} I(X) \ar@/_/[rddd]_-{} \ar@/^/[rrrd] ^-{\coh^{\ast}i_3 }\ar@{.>}[rd]_{}_-{\coh^{\ast} i} \\
    & \coh^{\ast}C \ar[rr]^-{\coh^{\ast}i_1 }\ar@{>>}[dd]^{\simeq}_{} \ar@/_/@{.>}[ul]_-{\coh^{\ast}h}  \cartesien& & \coh^{\ast}T(X) \ar@{>>}[dd]^{\simeq}_{} \\
    & & & \\
    & \coh^{\ast}I(S) \ar[rr]^-{\coh^{\ast}i_{2}}  & & \coh^{\ast}T(S) 
  }
$$
The map $\coh^{\ast}\overline{\overline{Tg}}$ is uniquely defined in the category of pointed $\coh^{\ast}R\otimes \coh^{\ast}T(S)-\coh^{\ast}T(S)\otimes\coh^{\ast}S $ graded bimodules sending the element (class) $1\otimes 1-1\in \coh^{\ast} I(S)$ to $1\otimes 1-1\in \coh^{\ast}I(X)$. More precisely, the chosen point in $\coh^{\ast}T(S)$  is the image of $1\otimes1-1\in \coh^{\ast}I(S)$ by $\coh^{\ast}i_{2}$. Obviously, it determines points in $\coh^{\ast}C$ and $ \coh^{\ast}TX$ in a canonical way. Notice that $\coh^{\ast}I(X)$ is already canonically pointed in  (cohomolgy class)  $1\otimes1-1$. We deduce that $\coh^{\ast}h:\coh^{\ast}C\rightarrow \coh^{\ast} I(X)$ is uniquely defined in the category of pointed  graded $\coh^{\ast}R\otimes \coh^{\ast}T(S)-\coh^{\ast}T(S)\otimes\coh^{\ast}S$-bimodule.     
In order to prove that $i: I(X)\rightarrow C$ is weak equivalence it is sufficient to prove that $\coh^{\ast}(i_3)\circ \coh^{\ast}(h)=\coh^{\ast}(i_ 1)$.  By construction $$i_1\circ \overline{Tg}:I(S)\rightarrow C\rightarrow T(X)$$ and  $$i_{3}\circ \overline{\overline{Tg}}: I(S)\rightarrow I(X)\rightarrow T(X)$$
coincide. Hence, 
$$i_{1}\circ \overline{Tg}\circ \overline{Tf}= i_{3}\circ \overline{\overline{Tg}}\circ \overline{Tf}$$ 
it implies that  
$$i_{1}\circ \overline{Tg}\circ \overline{Tf}= i_{3}\circ h.$$
On the other hand $\coh^{\ast}\overline{Tg}\circ \coh^{\ast} \overline{Tf}=id$, which implies that  $\coh^{\ast}(i_3)\circ \coh^{\ast}(h)=\coh^{\ast}(i_ 1)$. By universality of the pullback, we obtain that $\coh^{\ast}I(X)\rightarrow \coh^{\ast} C$ is an isomorphism hence $\overline{\overline{Tf}}: I(X)\rightarrow I(S)$  is an equivalence. 
We have the following commutative diagram of short exact sequences in $\Chk$
$$
  \xymatrix{
  I(X)\ar[r]^{inc}\ar[d]^{\sim} & T(X)\ar@{>>}[r]\ar[d]^{\sim} &F(X)\ar[d]^{}\\
  I(S)\ar[r]^{inc} & T(S)\ar@{>>}[r] & F(S)
   }$$
and by five-lemma (exact sequence with five terms), we finally conclude that the map $F(X)\rightarrow F(S)$ is a weak equivalence. 
\end{proof}

In order to verify the third axiom we have to prove that the derived functor of $F:\Chrs^{0}\rightarrow \Algrs$ preserves distinguished objects. 
\begin{lemma}[\textbf{Axiom III}]\label{maintheorem}
Let $R$ be a cofibrant DG algebra, and $S$ a DG algebra such that all homotopy invertible elements are strictly invertible. Then the left derived functor of $F:\Chrs^{0}\rightarrow \Algrs$ preserves distinguished objects. 
\end{lemma}
\begin{proof}
It is a consequence of Lemma \ref{S-eq} and Lemma \ref{lemma2}. 
\end{proof}
\section{The Mapping Space}

\subsection{The fundamental group of $\Map_{\Algk}(R,R)_{id}$}
 We denote the category of DG categories by $dg-\Cat$, the existence of model structure \` a la Dwyer-Kan, was initially proved in \cite{tabuada2005structure}. In \cite{toen}, To\"en gives a complete description of the mapping space of dg-categories. If $\RR$ is a cofibrant dg-category and $\SSS$ any dg-category, the mapping space $\Map_{dg-\Cat}(\RR,\SSS)$ is described as the nerve of the category of \textit{right quasi-representable} $\RR\otimes \SSS^{op}$-dg-functors \cite[Definition 4.1]{toen}. In the particular case where $\RR$ and $\SSS$ have just one object  with the underlying DG algebras $R$ and $S$ , then the right quasi-representable $\RR\otimes \SSS^{op}$-dg-functors, are exactly, the  $R-S$-bimodules which are (zig-zag) equivalent to $S$ as right $S$-module. In \cite[p.15]{DH2010}, such $R-S$-bimodules are called \textit{potentially distinguished} modules.
The nerve of the category of potentially distinguished modules is denoted by $\mathcal{M}_{R,S}$ and if $R=k$ we denote it simply by $\mathcal{M}_{S}$. Notice that if $R$ is cofibrant in $\Algk$ then $\RR$ is cofibrant in $dg-\Cat$. We summarize the previous discussion in the following Lemma. 
\begin{lemma}\label{moduli}
 The moduli space $\mathcal{M}_{R,S}$ is equivalent to $\Map_{dg-\Cat}(\RR,\SSS)$ and $\mathcal{M}_{S}$ is equivalent to $\Map_{dg-\Cat}(\mathsf{k},\SSS)$, where $\RR$ (resp. $\SSS$ and $\mathsf{k}$) is the dg-categorie with one object an the underlying endomorphim DG algebra $R$ (resp. $S$ and $k$).  
 \end{lemma}
 \begin{proof}
 It is a direct consequence of the computation of the mapping space in the model category of DG categories ($dg-\Cat$) cf. \cite[Theorem 4.2]{toen} and \cite[section 1.13]{DH2010}. 
 \end{proof}
 \begin{remark}\label{point}[\textbf{Base points}]
 If $\phi:R\rightarrow S$ is a morphism of DG algebras, the moduli space $\mathcal{M}_{R,S}$ is pointed at the object $S$ which is a canonical distinguished object equivalent to $S$ as a right $S$-module and has a structure of $R-S$-bimodule via $\phi$. The corresponding base point of $\Map_{dg-\Cat}(\RR,\SSS)$ is the morphism $\phi$. By the same way, the moduli space $\mathcal{M}_{S}$ is also pointed at the object $S$, and the corresponding base point of $\Map_{dg-\Cat}(\mathsf{k},\SSS)$ is the unit morphism $k\rightarrow S$. Finally, the base point of the space $\Map_{\Algk}(R,S)$ is the morphism $\phi$.
 \end{remark}
Now, we are ready to introduce the Dwyer-Hess fundamental Theorems \cite[Theorem 3.10 and 3.11]{DH2010} in the context of $\Algk$.
\begin{theorem}\label{Hess2}
Let $\phi: R\rightarrow S$ be a morphism of DG algebras (such that $S$ is a strict DG $k$-algebra \ref{lemma1bis}) . There exists a fiber sequence 
$$\Map_{\Algk}(R,S)\rightarrow \mathcal{M}_{R,S}\rightarrow\mathcal{M}_{S},$$
or equivalently, the fiber sequence:
$$ \Map_{\Algk}(R,S)\rightarrow\Map_{dg-\Cat}(\RR,\SSS)\rightarrow\Map_{dg-\Cat}(\mathsf{k},\SSS).$$

\end{theorem}

\begin{proof}
Since the categories $\Chk,~\Algk$ and $\Chrs$ verify the six Axioms (\ref{axiom6}) for any cofibrant DG $k$-algebra $R$ and any strict DG $k$-algebra $S$. Hence, there is a fiber sequence $\Map_{\Algk}(R,S)\rightarrow \mathcal{M}_{R,S}\rightarrow\mathcal{M}_{S}$ by \cite[Theorem 3.10]{DH2010}. For the second fiber sequence, we apply Lemma \ref{moduli}. 
\end{proof}
\begin{remark}
Notice that the homotopy limits in $\Algk$ and $dg-\Cat^{\ast} $ (where $dg-\Cat^{\ast}$ is the full subcategory of $dg-\Cat$ such that all categories have only one object) are the same and the fact that the mapping space commutes with homotopy limits, we have a  more general result: 
$$ \Map_{\Algk}(R,S)\rightarrow\Map_{dg-\Cat}(\RR,\SSS)\rightarrow\Map_{dg-\Cat}(\mathsf{k},\SSS).$$
is a fiber sequence for any cofibrant $R$ and any DG $k$-algebra $S$ which is homotopy limit of strcit DG $k$-algebras. Unfortunately, we don't know if any DG $k$-algebra is a homotopy limit of strict DG $k$-algebras.  
\end{remark}
\begin{theorem}\label{th2}
 Let $R$ be a strict  DG $k$-algebra, then there is an exact sequence of groups
$$\dots\rightarrow\coh^{-1} (R)\rightarrow  [R,R]_{1}^{\otimes}\rightarrow \HH^{0}(R)^{\star}\rightarrow \coh^{0}(R)^{\star}.$$
where $\HH^{0}(R)^{\star}$  is the group of units of $\HH^{0}(R)$ and  $\coh^{0}(R)^{\star}$ is the group of the units of the ring  $ \coh^{0}(R)$. 
 \end{theorem}
 \begin{proof}
By Theorem \ref{Hess2}, we have the fiber sequence
 \begin{displaymath}\label{a}
 \Map_{\Algk}(R,R)_{id}\rightarrow  \Map_{dg-\Cat}(\RR,\RR)_{id}\rightarrow\Map_{dg-\Cat}(\mathsf{k},\RR)_{\nu }.
 \end{displaymath}
Therefore, applying the long exact sequence of homotopy groups (starting at level one), we obtain  
 $$ \pi_{1}\Map_{dg-\Cat}(\mathsf{k},\RR)_{\nu}\leftarrow   \pi_{1}\Map_{dg-\Cat}(\RR,\RR)_{id}\leftarrow [R,R]_{1}^{\otimes}\leftarrow \pi_{2}\Map_{dg-\Cat}(\mathsf{k},\RR)_{id}\dots$$ 
By \cite[Corollary 8.3]{toen}, we have that $$\pi_{1}\Map_{dg-\Cat}(\RR,\RR)_{id}\simeq \HH^{0}(R)^{\star}.$$
By \cite[corollary 4.10]{toen}, we have that for $i>0$.
$$\pi_{1}\Map_{dg-\Cat}(k,\RR)_{\nu}\simeq \coh^{0}(R)^{\star}~~ \mathrm{and}~~ \pi_{i+1}\Map_{dg-\Cat}(k,\RR)_{\nu}\simeq \coh^{-i}(R).$$ 
  \end{proof}
\begin{corollary}
 Let $R$ be a DG algebra such that $\coh^{-1}(R)=0$, then 
 $$\pi_{1}\Map_{\Algk}(R,R)_{id}\simeq Ker[\HH^{0}(R)^{\star}\rightarrow \coh^{0}(R)^{\star}].$$
  \end{corollary}

\subsection{Higher Homotopy Groups of $\Map_{\Algk}(R,S)_{\phi}$} \label{higher} 
In this section, we give a complete description of the higher homotopy groups of the mapping space of the model category $\Algk$.  
\begin{theorem}\label{relation}
Let $\nu: k\rightarrow S$ be the unit, and $\phi: R\rightarrow S$ a map of $\Algk$ and $S$ is a strict DG $k$-algebra. There is a fiber sequence of spaces:
$$\Omega\Map_{\Algk}(R,S)_{\phi}\rightarrow \Map_{\Chrr}(R,S_{\phi})\rightarrow\Map_{\Chk}(k,S),$$
where $S_{\phi}$ is seen as $R$-bimodule via ${\phi}$.
\end{theorem}
\begin{proof}
It is a consequence of \ref{Hess2} and  \cite[Theorem 3.11]{DH2010}.
\end{proof}
\begin{theorem}\label{th1}
For any map of DG algebras $\phi: R\rightarrow S$ such that $S$ is a strict DG $k$-algebra, there is a long exact sequence of abelian groups  
$$ \coh^{-1}(S)\leftarrow \HH^{-1}(R,S)\leftarrow [R,S]^{\otimes}_{2}\leftarrow \coh^{-2}(S)\leftarrow \HH^{-2}(R,S)\leftarrow [R,S]^{\otimes}_{3}\leftarrow \dots  $$  
where $S$ is seen as an $R$-bimodule via $\phi$.
\end{theorem}
\begin{proof}
Let $\nu :k\rightarrow S$ be the unit morphism. We loop the previous fiber sequence \ref{relation} and obtain a new fiber sequence
\begin{equation}\label{eq}
\Omega^{2}\Map_{\Algk}(R,S)_{\phi}\rightarrow \Omega\Map_{\Chrr}(R,S_{\phi})\rightarrow\Omega\Map_{\Chk}(k,S).
\end{equation}
Since the model categories $\Chrr$ and $\Chk$ are stable and in particular pointed, the map 
$$\Omega\Map_{\Chrr}(R,S_{\phi})\rightarrow\Omega\Map_{\Chk}(k,S)$$
 induces a morphism of abelian groups
  $$\pi_{i} \Omega\Map_{\Chrr}(R,S_{\phi})\rightarrow\pi_{i}\Omega\Map_{\Chk}(k,S)~\textrm{for}~i\geq 0.$$ 
  By \ref{remarque1} and Definition \ref{Hochschild3}, there is an isomorphism of groups 
  $$\pi_{i}\Omega\Map_{\Chrr}(R,S_{\phi})\simeq \HH^{-1-i}(R,S)~\textrm{for}~i\geq 0,$$
   and by \ref{remarque2} $\pi_{i}\Omega\Map_{\Chk}(k,S)\simeq \coh^{-i-1}(S)$. Therefore, applying the Serre exact sequence to \ref{eq} we prove our theorem.


\end{proof}

\begin{corollary}\label{cor0}
If $\phi:R\rightarrow S$ a morphism of DG $k$-algebras such that $S$ is connective, then 
$$\HH^{-i}(R,S)\simeq  [R,S]^{\otimes}_{i+1}~\textrm{for all}~ i>0.$$
\end{corollary}
\begin{proof}
Since $S$ is connective, i.e., $\coh^{n}(S)=0$ for all strictly negative integers. According to the long exact sequence in \ref{th1}, $[R,S]^{\otimes}_{i+1}\simeq \HH^{-i}(R,S)$ for all $i>0$. 
\end{proof}


\section{Applications}
In all our application we assume that $S$ is a strict Dg $k$-algebra. 
\subsection{Infinity loop space}
The first evident consequence of \ref{relation} is the extra structure on the double loop space of $\Map_{\Algk}(R,S)$ which is summarized in the following result.
\begin{corollary}
Let $\phi:R\rightarrow S$ a map of DG algebras, such that $S$ is connective, then $\Omega^{2}_{\phi}\Map_{\Algk}(R,S)$ is an infinity loop space. 
\end{corollary} 
\begin{proof}
Since $\pi_{i}\Map_{\Chk}(k,S)$ vanishes for $i>0$ and $\Map_{\Chrr}(R,S_{\phi})$ is an infinity loop space because $\Chrr$ is a stable model category. Hence, by Theorem \ref{relation}, we conclude that
$$\Omega\Map_{\Chrr}(R,S_{\phi})\sim\Omega^{2}_{\phi}\Map_{\Algk}(R,S)$$
 is an infinity loop space. 
\end{proof}
\subsection{Derivations}
We make a connection with the theory of derivations of DG $k$-algebras. Let $R$ be in $\Algk$ and $M$ a DG $R$-bimodule. We define a new DG $R$-algebra $R\oplus M$, called a \textit{square zero extension} as follows. It is the DG algebras whose underlying complex is $R\oplus M$ and whose DG algebra structure is the obvious one 
induced from the trivial multiplication on $M$, i.e., $m.m^{'} = 0$ for any $m, m^{'}\in M.$ The map $\phi: R\rightarrow R\oplus M$ is the obvious map of DG $R$-algebras. 
In \cite{dugger2007topological}, the authors use the inverse gradation, i.e., \textbf{the differentials are of degree -1}. According to the long exact sequence described in \cite[3.14]{dugger2007topological} and their notations, if $M$ is coconnective then,
\begin{equation}\label{derr}
 \Derive^{-n}_{k}(R,M)\simeq \HH^{-n+1}(R,M)~\mathrm{for}~n>1. 
\end{equation}
\textbf{But with our gradation and notation} if $M$ is connective then $ \Derive^{-n}(R,M)\simeq \HH^{-n+1}(R,M)$ for all $n>1$, and we have the following lemma
\begin{lemma}\label{derivative}
If $R$ is a connective DG $k$-algebra, $R\oplus M$ is a connective square-zero extension graded differential $R$-bimodule, and $\phi:R\rightarrow R\oplus M$ the obvious inclusion, then  for all $n>1$
$$\Derive^{-n}_{k}(R,M)\oplus \HH^{-n+1}(R,R)\simeq \pi_{n}\Map_{\Algk}(R,R\oplus M)_{\phi}.$$
\end{lemma}
\begin{proof}
It is a consequence of \ref{cor0}, \ref{derr}, and the fact that the Hochschild cohomology is additive.  
\end{proof}


\subsection{Commutative DG algebras} Let
$k=\mathbb{Q}$ or any field of characteristic 0. The model category of commutative differential unbounded $k$-algebras is denoted by  $\CAlgk$ equipped with the induced model structure, i.e., weak equivalences are isomorphisms in homology (Cf. \cite[section 2.3.1]{toen2008homotopical}) and fibrations are degree-wise surjective morphisms. There is a Quillen adjunction 
$$ \xymatrix{ \Algk \ar@<2pt>[r]^{ Ab} & \CAlgk  \ar@<2pt>[l]^{U}},$$
where $Ab$ is called the abelianization functor. If $R\in\Algk$ is cofibrant and $S\in \CAlgk$, then by \cite[Theorem 2.12]{DH2010}, there is a weak homotopy equivalence of simplicial sets 
\begin{equation}\label{comm}
\Map_{\CAlgk}(Ab(R), S)\sim \Map_{\Algk}(R,S).
\end{equation}
Let $S\in \CAlgk$, and let $\phi: R\rightarrow S$ be a morphism of DG algebras, then it induces a morphism $\overline{\phi}: Ab(R)\rightarrow S$ in $\CAlgk$. 
\begin{corollary} Let $R$ be a cofibrant DG algebra, and let $S$ be a connective commutative  DG algebra with  $\phi: R\rightarrow S$ is a map of DG k-algebras. Then, there is an isomorphism of abelian groups
$$\pi_{i}\Map_{\CAlgk}(Ab(R), S)_{\overline{\phi}}\simeq \HH^{1-i}(R,S), ~ i>1. $$
\end{corollary}
\begin{proof}
It is a formal consequence of \ref{cor0} and the fact that $\Map_{\CAlgk}(Ab(R), S)_{\overline{\phi}}\sim \Map_{\Algk}(R,S)_{\phi}$ by \ref{comm}. 
\end{proof}
\begin{remark}
By the Eckmann-Hilton argument, the category $\CAlgk$ is the category of monoids in the monoidal category $(\Algk, \otimes)$ (Cf.\cite[Section 4]{porst2008categories}). It is tempting to apply Dwyer-Hess fundamental theorem for  $\CAlgk$ in order to compute $\Map_{\CAlgk}(R,S)$ for any commutative DG algebras $R$ and $S$. The problem is the $\textbf{Axiom III}$ (Cf. \ref{axiom3}), which is not verified in general for $\CAlgk$, otherwise it would mean that for any commutative DG algebra $R$, the natural map $Ab(R^{c})\rightarrow R$ is a weak equivalence in $\CAlgk$ (where, $R^{c}$ is a cofibrant replacement of $R$ in $\Algk$). An easy example, due to Lurie, is the free commutative algebra in two variables $R=\mathbb{Q}[x,y]$.  The cofibrant replacement of $R$ is the free associative DG algebra $R^{c}$ in three variables $x, y,~z$ such that $\mathrm{deg}(x)=\mathrm{deg}(y)=0$ and $dz=xy-yx$. if $S$ is any commutative DG algebra, by simple computation, we obtain $\pi_{0}\Map_{\CAlgk}(R,S)\simeq \coh^{0}(S)\oplus \coh^{0}(S)$, but  $\pi_{0}\Map_{\Algk}(R^{c},S)\simeq \coh^{0}(S)\oplus\coh^{0}(S)\oplus\coh^{-1}(S)$. We conclude that $Ab(R^{c})\rightarrow R$ is not an equivalence in general. 
\end{remark}

\subsection{Conclusion}
It is natural to ask the following questions:\\
\underline{\textit{Question 1:}} Are Theorems \ref{th2} and \ref{th1} still true if we replace $k$ by any commutative DG algebra $A$?\\
\underline{\textit{Question 2:}} What is the correct formulation of Theorems \ref{th2} and \ref{th1} in the setting of the stable monoidal model category of symmetric spectra $\mathsf{Sp}$ and the associated  category of ring spactra $\mathsf{Sp}^{\otimes}$?  \\

\textbf{Acknowledgement:} I'm grateful to Kathryn Hess for helpful discussions and pointing out some inconsistencies and imprecisions in the earlier version. I would like to thank Oriol Ravent\' os for explaining the sign conventions in the derived categories.

\bibliographystyle{plain} 
\bibliography{dganews}

\end{document}